\documentclass[12pt]{amsproc}
\usepackage{amssymb}
\usepackage{amsmath}
\usepackage{amsfonts}
\usepackage{geometry}
\usepackage{graphicx}
\usepackage{mathrsfs,amssymb}
\usepackage{bm}
\providecommand{\U}[1]{\protect\rule{.1in}{.1in}}
\theoremstyle{plain}

\newtheorem{lemma}{Lemma}

\newtheorem{theorem}{Theorem}
\numberwithin{equation}{section}

\begin{document}
\title{On Kakeya-Nikodym type maximal inequalities}
\author{Yakun Xi}
\address{
Department of Mathematics\\
Johns Hopkins University\\
Baltimore, MD 21218, USA\\
Emails:  ykxi@math.jhu.edu }
\date{}
\subjclass{42B25} \keywords {Kakeya maximal function, Nikodym maximal function, Geometric combinatorics.}
\dedicatory{ }

\begin{abstract}
We show that for any dimension $d\ge3$, one can obtain Wolff's $L^{(d+2)/2}$ bound on Kakeya-Nikodym maximal function in $\mathbb R^d$ for $d\ge3$ without the induction on scales argument. The key ingredient is to reduce to a 2-dimensional $L^2$ estimate with an auxiliary maximal function. We also prove that the same $L^{(d+2)/2}$ bound holds for Nikodym maximal function for any manifold $(M^d,g)$ with constant curvature, which generalizes Sogge's results for $d=3$ to any $d\ge3$. As in the 3-dimensional case, we can handle manifolds of constant curvature due to the fact that, in this case, two intersecting geodesics uniquely determine a 2-dimensional totally geodesic submanifold, which allows the use of the auxiliary maximal function.
\end{abstract}
\maketitle
\section{Introduction}
In this paper, we reprove Wolff's $L^{(d+2)/2}$ bound on Kakeya-Nikodym maximal function \cite{wolff} in $\mathbb R^d$ for $d\ge3$ without appealing to induction on scales, by reducing to the 2-dimensional $L^2$ estimate with an auxiliary maximal function. The main argument is a modification of Sogge's work \cite{sogge}. By using a similar strategy and some geometric observations, we are also able to show the same bound holds for Nikodym maximal function on manifold $(M^d,g)$ with constant curvature, which is a generalization of Sogge's work on 3-dimensional case in \cite{sogge}. As in Sogge's work \cite{sogge} for 3-dimensional case, we can handle manifolds of constant curvature due to the fact that, in this case, two intersecting geodesics uniquely determine a 2-dimensional totally geodesic submanifold, which allows the use of the auxiliary maximal function.

The original Kakeya problem, proposed by Kakeya \cite{kakeya} in 1917, is to determine the minimal area needed to continuously rotate a unit line segment in the plane by 180 degrees. In 1928, Besicovitch \cite{besicovitch} showed that such sets may have arbitrary small measure. Moreover, he also constructed subsets of $\mathbb R^d$ of measure zero which contain a unit line segment in every direction. Such sets are called Besicovitch sets or Kakeya sets. The Kakeya conjecture states that any Besicovitch sets in $\mathbb R^d$ must have (Hausdorff or Minkowski) dimension $d$. 

The so-called maximal Kakeya conjecture is actually a stronger one that involves the following Kakeya maximal function
\[f_\delta^*(\xi)=\sup\limits_{a\in \mathbb R^{d}}\frac{1}{|T_\xi^\delta(a)|}\int_{T_\xi^\delta(a)}|f(y)|dy,\]
where $T_\xi^\delta(a)$ is a $1\times\delta\times\cdots\times\delta$ tube centered at $a\in\mathbb R^d$ with direction $\xi\in S^{d-1}$. This maximal conjecture(formulated by Bourgain \cite{bourgain}) says for any $\epsilon>0$
\begin{equation}
\label{I1}
\|f_\delta^*\|_{L^d(S^{d-1})}\le C_\epsilon \delta^{-\epsilon}\|f\|_{L^d(\mathbb R^d)}.
\end{equation}
Interpolating with the trivial $L^1\rightarrow L^\infty$ bound, we see \eqref{I1} is equivalent to
\begin{equation}
\label{I2}
\|f_\delta^*\|_{L^q(S^{d-1})}\le C_\epsilon\delta^{1-\frac{d}{p}-\epsilon}\|f\|_{L^p(\mathbb R^d)},
\end{equation}
where $1\le p\le d$ and $q=(d-1)p'$.

It is well-known(see Lemma 2.15 in \cite{bourgain} for details) that an estimate like \eqref{I2} for a given $p$ would imply that Kakeya sets have (both Hausdorff and Minkowski) dimension at least $p$.
For the case $d=p=2$, \eqref{I1} was proved by Cordoba \cite{cordoba}. However, it is still open for any $d\ge 3$. When $p=(d+1)/2,\ q=(d-1)p'=d+1$, \eqref{I2} follows from Drury \cite{drury} in 1983. In 1991, Bourgain \cite{bourgain} improved this result for each $d\ge3$ to some $p(d)\in((d+1)/2,(d+2)/2)$ by the so-called bush argument, where Bourgain considered the \textquotedblleft{bush}\textquotedblright{} where lots of tubes intersect at a given point. Four years later, Wolff \cite{wolff} generalized Bourgain's bush argument to the hairbrush argument, by considering tubes with lots of \textquotedblleft bushes\textquotedblright{} on them. Combining this hairbrush argument and the induction on scales, Wolff further improved Bourgain's result. Moreover, Wolff also pointed out that the same proof applies to the closely related Nikodym maximal function:
\[f_\delta^{**}(x)=\sup\limits_{x\in\gamma_x}\frac{1}{|T_{\gamma_x}^\delta|}\int_{T_{\gamma_x}^\delta}|f(y)|dy,\]
where $\gamma_x$ denotes the unit line segments that contains the point $x$. It is well-known that a bound like \eqref{I2} for Nikodym maximal function would also imply a corresponding lower bound for the dimension of the compliment of the Nikodym sets. Wolff \cite{wolff} proved the following bound for both Kakeya and Nikodym maximal functions.
\begin{theorem}$\mathrm{(T.\ Wolff,\ 1995)}$\label{THM1}The Kakeya maximal function satisfies
\begin{equation}
\label{I3}
\|f_\delta^*\|_{L^\frac{(d-1)(d+2)}{d}(S^{d-1})}\le C_\epsilon\delta^{1-\frac{2d}{d+2}-\epsilon}\|f\|_{L^{\frac{d+2}{2}}(\mathbb R^d)}.
\end{equation}
Similarly, for the Nikodym maximal function, we have
\begin{equation}
\label{I4}
\|f_\delta^{**}\|_{L^\frac{(d-1)(d+2)}{d}(\mathbb R^d)}\le C_\epsilon\delta^{1-\frac{2d}{d+2}-\epsilon}\|f\|_{L^{\frac{d+2}{2}}(\mathbb R^d)}.
\end{equation}
\end{theorem}
As mentioned before, \eqref{I3} implies that the Hausdorff dimension of a Kakeya set is at least $(d+2)/2$. This is still the best result for the (Hausdorff) Kakeya conjecture when $d=3, 4$. One can get better results for larger $d$ or for Minkowski dimension, see e.g. \cite{bourgain2}, \cite{tao}, \cite{tao2}.

It is easy to see that one can naturally extend the definition of the Nikodym maximal function to manifolds. In 1997, Minicozzi and Sogge \cite{soggem} showed for a general manifold, Drury's result where $p=(d+1)/2$ still holds, but surprisingly, they constructed some examples to show that it is indeed sharp in odd dimensions. In 1999, Sogge \cite{sogge} managed to adapt Wolff's method for the generalized Nikodym maximal function to 3-dimensional manifolds with constant curvature. Combining a modified version of Wolff's multiplicity argument with an auxiliary maximal function, Sogge proved the following
\begin{theorem}$\mathrm{(C.\ Sogge,\ 1999)}$ \label{THM2}Assume that $\mathrm(M^3,g\mathrm)$ has constant curvature. Then for $f$ supported in a compact subset $K$ of a coordinate patch and all $\epsilon>0$
\begin{equation}
\|f_\delta^{**}\|_{L^\frac{10}{3}(M^3)}\le C_\epsilon \delta^{-\frac{1}{5}-\epsilon}\|f\|_{L^\frac{5}{2}(M^3)}.
\end{equation}
\end{theorem}

In his proof, Sogge was able to avoid the induction on scales argument, which is hard to perform in curved space.
In 2013, using Sogge's method, Miao, Yang and Zheng \cite{miao} reproved Wolff's result for Kakeya maximal function in $\mathbb R^3$ without appealing to induction on scales. Indeed, they tried to recover Wolff's bound for any dimension $d\ge3$ by reducing to $(d-1)$-dimensional $L^2$ estimate for the auxiliary maximal function, which would be an induction on dimensions argument that is similar to Bourgain's argument in \cite{bourgain}. Unfortunately, there is a $\delta^{-(d-3)/2}$ loss in the bound for the auxiliary maximal function, which basically prevents one from getting Wolff's bound if $d\neq 3$. 

Our paper is organized as follows. In the first half, we modify Sogge's strategy to show that if we add in some more geometric observations, we can get rid of the $\delta^{-(d-3)/2}$ loss for the auxiliary maximal function, by just reducing to Cordoba's \cite{cordoba} optimal $L^2$ estimate for 2-planes. This modification helps us to recover Wolff's result. In the second half, we adapt the same idea to the Nikodym-type maximal function in the constant curvature case, and extend Sogge's result \cite{sogge} to any dimension $d\ge3$, where we shall of course need a curved version of the optimal $L^2$ estimate for Nikodym maximal function which is due to Mockenhaupt, Seeger and Sogge \cite{soggeg}.

\section{Kakeya maximal function in Euclidean space}
In this section, we reprove \eqref{I3} without appealing to induction on scales. We shall follow the strategy in \cite{sogge} and \cite{miao} closely, and add in some key observations. Throughout this section, we use $C$, $c$ to denote various constants that only depend on the dimension.
\subsection{Preliminaries}
It is well-known that it suffices to prove the following restricted weak type estimate:
\begin{equation}\label{eqK1}
|\{\xi\in S^{d-1}:(\chi_E)_\delta^*(\xi)\ge\lambda\}|\lesssim_\epsilon(\lambda^{-p}\delta^{p-d}|E|)^{\frac{q}{p}},
\end{equation}
where $E$ is contained in the unit ball, $\chi_E$ denotes its characteristic function, $p=\frac{d+2}{2}$ and $q=\frac{(d-1)p}{p-1}$. For the sake of simplicity, we use the notation $A\lesssim_\epsilon B$ throughout the paper to denote $A\le C_\epsilon \delta^{-\epsilon}B$. Similarly, $B\gtrsim_\epsilon A$ means $B\ge c_\epsilon \delta^{\epsilon}A$. 

We start by doing some standard reductions(see e.g. \cite{bourgain}). First, without loss of generality, we can assume that any $\xi, \xi'\in \{\xi\in S^{d-1}:(\chi_E)_\delta^*(\xi)\ge\lambda\} $ have angle $\angle(\xi,\xi')\le1$. Second, we take a maximal $\delta$-separated subset $\{\xi^i\}_{i=1}^M$ of $\{\xi\in S^{d-1}:(\chi_E)_\delta^*(\xi)\ge\lambda\}$, then \eqref{eqK1} is equivalent to

\begin{equation}\label{eqK2}
M\delta^{d-1}\lesssim_\epsilon(\lambda^{-p}\delta^{p-d}|E|)^{\frac{q}{p}},
\end{equation}
which is equivalent to
\begin{equation}\label{eqK3}
|E|^2\gtrsim_\epsilon\lambda^{d+2}\delta^{d-2}(M\delta^{d-1})^\frac{d}{d-1}.\end{equation}
For each $\xi^i$, there is a tube $T_{\xi^i}^\delta:=T_i^\delta$ satisfying 
\begin{equation}|E\cap T_i^\delta|\ge\lambda|T_i^\delta|.\label{eqK4}
\end{equation}
{\bf Remark:} Indeed, we will always assume $\lambda\ge\delta$ in proving \eqref{eqK3}, for the reason that in the case $\lambda\le\delta$, it's trivial that $|E|^2\ge|E\cap T_i^\delta|^2\ge\lambda^2\delta^{2d-2}\ge\lambda^{d+2}\delta^{d-2}\gtrsim\lambda^{d+2}\delta^{d-2}(M\delta^{d-1})^\frac{d}{d-1}$. The last inequality follows from the simple fact $M\delta^{d-1}\lesssim 1$.

We start our proof by applying a multiplicity argument to these tubes, which was first introduced by Wolff. We will be using a strengthened version developed by Sogge, see Lemma 2.5 in \cite{sogge}. This modification by Sogge is crucial if one wants to avoid induction on scales.

\subsection{Multiplicity argument}

Consider parameters $\theta\in[\delta, 1],$ $\sigma \in [\lambda\delta, 1]$.
First, for $1\le j\le M$ and $x\in T_j^\delta$ fixed, let
$$\mathfrak{L}_\theta(x,j)=\{i:x\in T_i^\delta, \angle{(T_j^\delta, T_i^\delta)\in[\theta/{2},\theta)}\}$$
index the tubes $T_i^\delta$ containing $x$ which intersect the fixed tube $T_j^\delta$ at angle comparable to $\theta$. Next, let
$$\mathfrak{L}_\sigma(x,j)=\{i:x\in T_i^\delta, |T_i^\delta\cap\{y\in E: \mathrm{dist}(y,\gamma_j)\in[\sigma/{2},\sigma)\}|\ge(2\log_2\frac{1}{\delta^2})^{-1}\lambda|T_i^\delta|\}$$
index the tubes $T_i^\delta$ containing $x$ which intersect the fixed tube $T_j^\delta$ at $x$ such that there is a non-trivial portion of $T_i^\delta\cap E$ that has distance to $\gamma_j$ comparable to $\sigma$.
Now let
$$\mathfrak{L}_{\theta,\sigma}(x,j)=\mathfrak{L}_{\theta}(x,j)\cap\mathfrak{L}_{\sigma}(x,j),$$
then we have the following
\begin{lemma}\label{lemma1}There are $N\in\mathbb{N}$ and $\theta\in[\delta, 1],$ $\sigma \in [\lambda\delta, 1]$ that fulfill the following two cases

$\mathrm{I}.$ $($Low multiplicity case$)$There are at least $M/2$ values of $j$ for which
\[\left|\left\{x\in T_j^\delta\cap E: \#\{i:x\in T_i^\delta\}\le N\right\}\right|\ge\frac{\lambda}{2}|T_j^\delta|.\]

$\mathrm{II}_{\theta,\sigma}.$ $($High multiplicity case at angle $\theta$ and distance $\sigma$$)$There are at least $M/(2(\log_2{1/\delta^2}))^2$ many values of $j$ for which
\begin{equation}\label{eqK5}\left|\left\{x\in T_j^\delta\cap E: \#\mathfrak{L}_{\theta,\sigma}\ge\frac{N}{(2\log_2\frac{1}{\delta^2})^2}\right\}\right|\ge\frac{\lambda}{(4\log_2\frac{1}{\delta^2})^2}|T_j^\delta|.\end{equation}
\end{lemma}
\begin{proof} Choose the smallest $N\in\mathbb{N}$ that satisfies the low multiplicity case I. Then there must be $M/2$ values of $j$ such that
\begin{equation}\label{eqK6}|\{x\in T_j^\delta\cap E: \#\{i:x\in T_i^\delta\}\ge N\}|\ge\frac{\lambda}{2}|T_j^\delta|.\end{equation}
We claim that for any such fixed $j$ and $x\in T_j^\delta\cap E$ with $\#\{i:x\in T_i^\delta\}\ge N$ we can find $1\le m_{x,j}\le\log_2\frac{1}{\delta}$, and $1\le n_{x,j}\le\log_2\frac{1}{\lambda\delta}\le\log_2\frac{1}{\delta^2}$ such that
\[\#\mathfrak{L}_{2^{m_x}\delta,2^{n_x}\lambda\delta}(x,j)\ge \frac{N}{(2\log_2\frac{1}{\delta^2})^2}.\]
Indeed, if the inequality fails for every pair of such $(m,n)$, summing over them would give us a contradiction. Similarly, for a fixed $j$, using the pigeonhole principle again, we can find some uniform $1\le m_j\le\log_2\frac{1}{\delta}$ and $1\le n_j\le\log_2\frac{1}{\delta^2}$ such that \eqref{eqK5} holds for all such fixed $j$.  Finally, since there are $M/2$ values of $j$ satisfying \eqref{eqK6}, if we use pigeonhole principle one more time, we conclude that we can choose fixed $\theta=2^m\delta, \sigma=2^n\lambda\delta$, so that \eqref{eqK5} holds for at least $M/(2(\log_2{1/\delta^2}))^2$ many values of $j$.
\end{proof}
{\bf Remark:} The reason that we need $\sigma$ to go down to the scale $\lambda\delta$ instead of $\delta$ is that we only have $\lambda|T_j^\delta|$ portion of each $T_j^\delta$ to apply pigeonhole principle, but this does not hurt us thanks to the fact that $\lambda\ge\delta$. Furthermore, noting that for such $\theta, \sigma$ that fulfill $\mathrm{II}_{\theta,\sigma}$, we must have
\begin{equation}\label{eqK12}\lambda\lesssim_\epsilon\frac{\sigma}{\theta}\lesssim 1.\end{equation}
This will be crucial to extend \cite{miao} to any dimension.

\subsection{Auxiliary maximal function}
First we prove a simple geometric lemma which will be useful in our proof and can be easily generalized to the constant curvature setting.

\begin{figure}
  \centering
    \includegraphics[height=8cm]{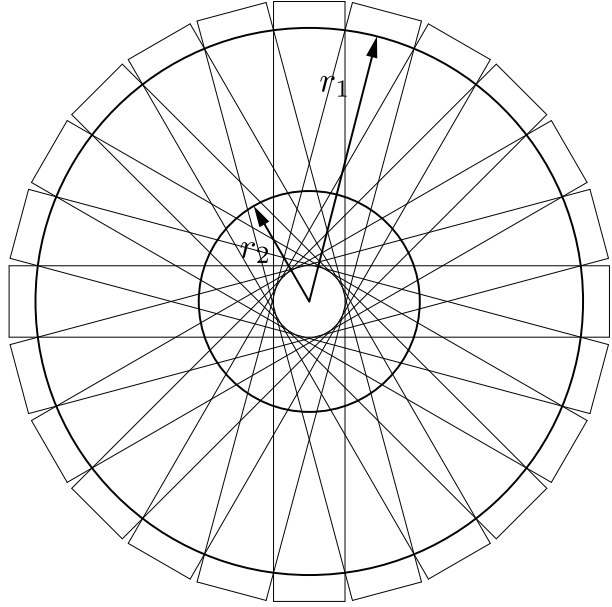}
  \caption{The overlapping of $\{l_k^\delta\}$.}
  \label{3}
\end{figure}

\begin{lemma}\label{lemma2} Let $0<r_2\le r_1<1$, and take a maximal $\delta$-separated subset $\{v_k\}$ on $r_1S^{d-2}$. Let $l_k^\delta$ be the $\delta$-neighborhood of the line passing through the origin with direction $v_k$, then the number of overlaps of $\{l_k^\delta\}$ at some point $y\in r_2S^{d-2}$ is at most
$$C\left(\frac{r_1}{r_2}\right)^{d-2},$$
which implies
$$\sum\limits_k\chi_{l_k^\delta\cap \{y':|y'|\in[r_2/2,r_2)\}}(y')\lesssim \left(\frac{r_1}{r_2}\right)^{d-2}.$$
\end{lemma}
\begin{proof}See Figure \ref{3}. Since the points $\{r_2v_k\}$ will be $\frac{r_2\delta}{r_1}$-separated on $r_2 S^{d-2}$, the number of overlaps of $\{l_k^\delta\}$ is bounded by 
$$C\frac{\delta^{d-2}}{(\frac{r_2\delta}{r_1})^{d-2}}\sim C\left(\frac{r_1}{r_2}\right)^{d-2},$$ hence the lemma.
\end{proof}
{\bf Remark:} It is easy to extend this result to manifolds with constant curvature. One just needs the simple observation that for two geodesics $\gamma_1(s),\gamma_2(s)$ parametrized by arc length, that satisfy $\gamma_1(0)=\gamma_2(0)$ and $\angle(\gamma_1,\gamma_2)=\beta$, then the distance $l(r)$ between $\gamma_1(r)$ and $\gamma_2(r)$ would satisfy 
$$cr\beta \le l(r)\le Cr\beta,$$
where $c, C$ only depend on the curvature, providing $r\le\min\{1, \frac{1}{2}(\mathrm{injectivity\ radius})\}$.

Within this section, we fix $j$ and consider the tube $\bm T^\delta=T_{\xi^j}^\delta$. We may assume without loss of generality that the central axis $\gamma_j$ of $\bm T^\delta$ is parallel to $e_1$, where $\{e_1,e_2,\ldots,e_d\}$ is an orthogonal normal basis of $\mathbb{R}^d$. For $y\in\mathbb{R}^d$, denote $y=(y_1,y')=(y_1,y_2,y''), \xi=(\xi_1,\xi')=(\xi_1,\xi_2,\xi'')$, where $y',\xi'\in\mathbb{R}^{d-1}, y'',\xi''\in\mathbb{R}^{d-2}$ respectively.

We define the auxiliary maximal function as

$$A_{\delta}^{\theta,\sigma}(f)(\xi)=\sup\limits_{T_\xi^\delta:\bm T^\delta\cap T_\xi^\delta\neq\emptyset, \angle(\bm T^\delta,T_\xi^\delta)\in[\theta/2,\theta)}\frac{1}{|T_\xi^\delta|}\int_{T_\xi^\delta\cap \{y:|y'|\in[\sigma/2,\sigma]\}}|f(y)|dy,$$
and define $A_{\delta}^{\theta,\sigma}(f)(\xi)$ to be zero if $\angle(\bm T^\delta,T_\xi^\delta)$ is outside the interval $[\theta/2,\theta)$.
\begin{theorem}Let $A_{\delta}^{\theta,\sigma}$ as above, then we have
\begin{equation}\label{eqK7}\|A_{\delta}^{\theta,\sigma}(f)\|_{L^2(S^{d-1})}\lesssim\left(\log\frac{1}{\delta}\right)^\frac{1}{2}\left(\frac{\theta}{\sigma}\right)^\frac{d-2}{2}\|f\|_{L^2(\mathbb{R}^d)}.\end{equation}
\label{theorem1}\end{theorem}
\begin{proof} Write $A_{\delta}^{\theta,\sigma}(f)$ simply as $A(f)$. Clearly, it suffices to estimate the integral
$$\int_{S_+^{d-1}}|A(f)|^2(\xi)dS,$$
where $S_+^{d-1}$ is the half-sphere $\{\xi\in S^{d-1}:\xi_1\ge0\}, $ and $dS$ is the corresponding surface measure.

Since $\angle(\xi,e_1)\in[\theta/2,\theta)$, we see that $\sin\theta/2\le|\xi'|<\sin\theta.$ Let
$$C_\theta=\{\xi'\in\mathbb{R}^{d-1}:\sin\theta/2\le|\xi'|<\sin\theta\}.$$

\begin{figure}
  \centering
    \includegraphics[height=8cm]{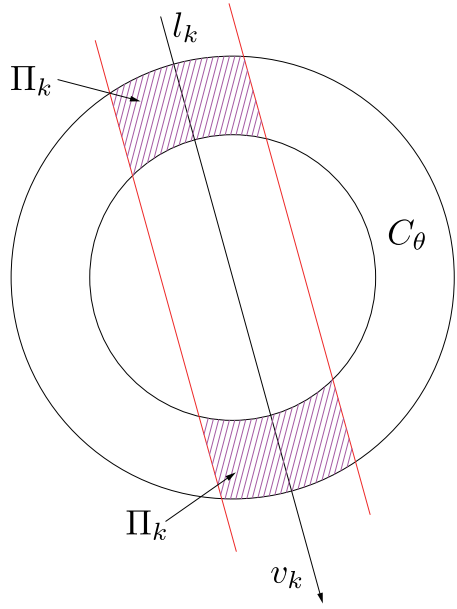}
  \caption{$\Pi_k$ in $\mathbb R^{d-1}$.}
  \label{1}
\end{figure}

Take a maximal $\frac{\delta}{\sin\theta}$-separated subset $\{v_k\}$ of $S^{d-2}$, which has size comparable to $[(\frac{\theta}{\delta})^{d-2}]$. Let $l_k$ be the line passing through the origin with direction $v_k$, and $l_k^\delta$ denotes the $\delta$-neighborhood of $l_k$. Let $\Pi_k=l_k^\delta\cap C_\theta,$ and note that $\{\sin\theta\cdot v_k\}$ is a maximal $\delta$-separated subset in $\sin\theta\cdot S^{d-2}$, so we must have $\cup_k\Pi_k\supset C_\theta$. Again by the maximality of $\{v_k\}$, we see that $\{\Pi_k\}$ has bounded overlap, so they are essentially pairwise disjoint. Indeed, we can take a new collection of sets $\{\Gamma_k\}$ which also covers $C_\theta$, with $\Gamma_1=\Pi_1$, and $\Gamma_k=\Pi_k\setminus\cup_{j=1}^{k-1}\Pi_j$. Clearly each $\Gamma_k$ is nonempty and they are pairwise disjoint.

Taking $r_1=\theta\sim\sin\theta, r_2=\sigma$ in Lemma \ref{lemma1}, we see that
$$\sum\limits_k\chi_{l_k^\delta\cap \{y':|y'|\in[\sigma/2,\sigma)\}}(y')\lesssim \left(\frac{\theta}{\sigma}\right)^{d-2}.$$

Consider $\xi'\in\Gamma_k$ for some $k$. Remember that we require $\bm T^\delta\cap T_\xi^\delta\neq\emptyset$, so the tube $T_\xi^\delta$ with direction $\xi=(\sqrt{1-|\xi'|^2},\xi')$ must lie in a 10$\delta$-neighborhood $H_k^{10\delta}$ of the 2-plane 
$$H_k=\textrm{span}\{e_1,(0,v_k)\},$$
see Figure \ref{2}.
Let
$$V_k=\left\{y\in\mathbb{R}^d:|y_1|\le1\right\}\cap H_k^{10\delta},$$
then clearly
$$\sum\limits_k\chi_{V_k\cap \{y:|y'|\in[\sigma/2,\sigma)\}}(y)\lesssim\sum\limits_k\chi_{l_k^\delta\cap \{y':|y'|\in[\sigma/2,\sigma)\}}(y')\lesssim \left(\frac{\theta}{\sigma}\right)^{d-2}.$$

\begin{figure}
  \centering
    \includegraphics[height=8cm]{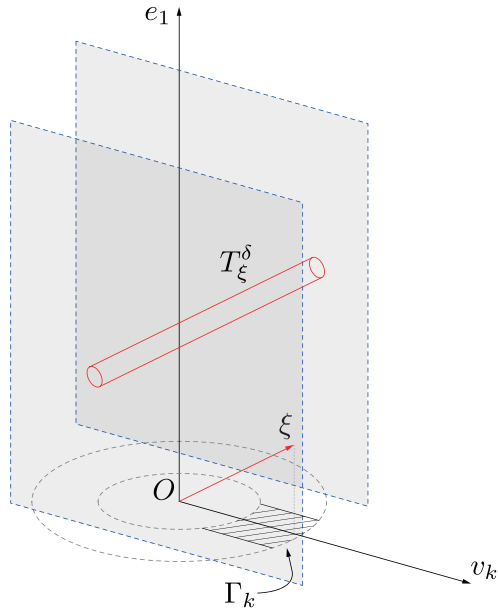}
  \caption{$T_\xi^\delta$ contained in $V_k$}
  \label{2}
\end{figure}

Now we begin to estimate $\int_{S_+^{d-1}}|A(f)|^2(\xi)dS.$ We claim that it suffices to prove the following $L^2$ estimate for each $V_k$,
\begin{equation}\label{eqK8}\|A(f\chi_{V_k})\|_{L^2(\{\xi\in S_+^{d-1}:\xi'\in\Gamma_k\})}\lesssim\left(\log\frac{1}{\delta}\right)^{\frac{1}{2}}\|f\chi_{V_k}\|_{L^2}.\end{equation}
Indeed, noting $\theta\le1$,
$$\begin{aligned}
\int_{S_+^{d-1}}|A(f)|^2(\xi)dS&\lesssim 4\int_{\mathbb R^{d-1}}|A(f)|^2(\sqrt{1-|\xi'|^2},\xi')d\xi'\\
&\lesssim\sum\limits_k\int_{\Gamma_k}|A(f\chi_{V_k})|^2(\sqrt{1-|\xi'|^2},\xi')d\xi'\\
&\lesssim\left(\log\frac{1}{\delta}\right)\sum\limits_k\int_{\mathbb{R}^d}|(f\chi_{V_k})|^2dy\\
&\lesssim\left(\log\frac{1}{\delta}\right)\left(\frac{\theta}{\sigma}\right)^{d-2}\|f\|_2^2.
\end{aligned}$$

Now we prove \eqref{eqK8}. Without loss of generality, assume $(0,v_k)=e_2$, and only consider functions $f$ with support in $V_k\cap\{y:|y'|\in[\sigma/2,\sigma)\}$. 

Let $\mathbf{P}(y'')$ be the 2-plane parallel to $\textrm{span}\{e_1,e_2\}=H_k,$ where $|y''|<10\delta$ and $y''$ is the $(d-2)$-dimensional parameter that determines the position of $\mathbf{P}(y'')$, in other word, the 2-plane $\mathbf{P}(y'')$ passes through the point $(0,0,y'')$.
For any $\xi'\in\Gamma_k, \xi=(\xi_1,\xi')$, the intersection $\mathbf{P}(y'')\cap T_\xi^\delta$ is the intersection of a 2-plane with a $d$-dimensional $\delta$-tube, so clearly it can always be contained in some 2-dimensional tube $t^\delta(y'')$ with direction $\frac{(\xi_1,\xi_2)}{\sqrt{\xi_1^2+\xi_2^2}}$.

Take $r=\sqrt{1-|\xi''|^2}$ then $r\sim\sqrt{1-C\delta^2}\ge\frac{1}{2}$, and let $M_\delta$ be the standard 2-dimensional Kakeya maximal function. Then we have
$$\begin{aligned}
\delta^{-(d-1)}\int_{T_\xi^\delta}|f(y)|dy&=\delta^{-(d-1)}\int_{|y''|\le10\delta}dy''\int_{\mathbf{P}(y'')\cap T_\xi^\delta}|f(y_1,y_2,y'')|dy_1dy_2\\
&\le\delta^{-(d-1)}\int_{|y''|\le10\delta}dy''\int_{t^\delta(y'')}|f(y_1,y_2,y'')|dy_1dy_2\\
&\lesssim\delta^{-(d-2)}\int_{|y''|\le10\delta}M_\delta(f(\ldots,y''))\left(\frac{\sqrt{r^2-|\xi_2|^2},\xi_2}{r}\right)dy'',
\end{aligned}$$
therefore,
$$A(f)(\xi)\lesssim\delta^{-(d-2)}\int_{|y''|\le10\delta}M_\delta(f(\ldots,y''))\left(\frac{\sqrt{r^2-|\xi_2|^2},\xi_2}{r}\right)dy''.$$
Noticing that if $\phi$ is some proper parameter for the subset of $S^1$ where $|\xi_2|\le\sin\theta\le\sin1$, then $|\frac{d\phi}{d\xi_2}|$ is bounded by some constant. Minkowski's inequality gives us
$$\begin{aligned}
&\ \ \ \ \left(\int_{|\xi_2|\le\sin\theta}|A(f)(\xi')|^2d\xi_2\right)^\frac{1}{2}\\&\lesssim\delta^{-(d-2)}\int_{|y''|\le10\delta}dy''\left(\int_{|\xi_2|\le\sin\theta}|M_\delta(f(\ldots,y''))|^2\left(\frac{\sqrt{r^2-|\xi_2|^2},\xi_2}{r}\right)d\xi_2\right)^{\frac{1}{2}}\\
&\lesssim\delta^{-(d-2)}\int_{|y''|\le10\delta}dy''\left(\int_{S^1}|M_\delta(f(\ldots,y''))|^2(\phi)d\phi\right)^{\frac{1}{2}}\\
&\lesssim\left(\log\frac{1}{\delta}\right)^\frac{1}{2}\delta^{-(d-2)}\int_{|y''|\le10\delta}\|f(\ldots,y'')\|_{L^2(y_1,y_2)}dy''\\
&\lesssim\left(\log\frac{1}{\delta}\right)^\frac{1}{2}\delta^{-\frac{(d-2)}{2}}\|f\|_{L^2}.
\end{aligned}$$
Therefore,
$$\begin{aligned}\left(\int_{\{\xi\in S_+^{d-1}:\xi'\in\Gamma_k\}}|A(f)(\xi)|^2dS\right)^\frac{1}{2}&\lesssim\left(\int_{\Gamma_k}|A(f)|^2(\sqrt{1-|\xi'|^2},\xi')d\xi'\right)^\frac{1}{2}\\
&\le\left(\int_{\{\xi':|\xi_2|\le\sin\theta,|\xi''|\le10\delta\}}|A(f)|^2(\sqrt{1-|\xi'|^2},\xi')d\xi'\right)^\frac{1}{2}\\
&=\left(\int_{|\xi''|\le10\delta}d\xi''\int_{|\xi_2|\le\sin\theta}|A(f)(\xi')|^2d\xi_2\right)^\frac{1}{2}\\
&\lesssim\left(\log\frac{1}{\delta}\right)^\frac{1}{2}\delta^{-\frac{(d-2)}{2}}\left(\int_{|\xi''|\le10\delta}\|f\|^2_{L^2}d\xi''\right)^\frac{1}{2}\\
&\lesssim\left(\log\frac{1}{\delta}\right)^\frac{1}{2}\|f\|_{L^2}.
\end{aligned}$$
This finishes the proof of \eqref{eqK8}, hence \eqref{eqK7} is proved.
\end{proof}
{\bf Remark:} The key difference between our auxiliary maximal function estimate and that in \cite{miao} is that we reduce to the optimal $2$-dimensional $L^2$ Kakeya bound for 2-planes rather than reducing to $(d-1)$-dimensional case for hyperplanes. In this way, instead of a $\delta^{-(d-3)/2}$ loss, the extra factor $(\theta/\sigma)^{(d-2)/2}$ we have can be handled using \eqref{eqK12}. This is actually natural if one looks back to Wolff's original hairbrush argument, the 2-dimensional $L^2$ estimate for 2-planes is enough to justify that the \textquotedblleft{bristles}\textquotedblright{} are essentially separated. In other words, reducing to 2-dimensional case already gives the best possible result for the hairbrush argument, so we don't expect improvements by reducing to $(d-1)$-dimensional case.

\subsection{A key lemma}
From now on, let N be the number that fulfills both case I and $\mathrm{II}_{\theta,\sigma}$, and again we fix an index $j$ such that $\bm T^\delta=T_{\xi^j}^\delta$ satisfies $\mathrm{II}_{\theta,\sigma}$. Using our $L^2$ estimate for the auxiliary maximal function, we will show that we can generalize Proposition 2.5 in \cite{sogge} and Lemma 5.2 in \cite{miao} to any dimension $d\ge3$, which was the part where Wolff needed induction on scales in his paper.
\begin{lemma}\label{lemma3}For any $\epsilon>0$, any point $a$
\begin{equation}\label{eqK9}|E\cap B(a,\delta^\epsilon\lambda)^c\cap \bm T^\sigma|\gtrsim_\epsilon\lambda^d N\sigma\delta^{d-2}.
\end{equation}
\end{lemma}
\begin{proof} We claim that it suffices to show
\begin{equation}\label{eqK10}|E\cap \bm T^\sigma|\gtrsim_\epsilon\lambda^d N\sigma\delta^{d-2}.\end{equation}

Indeed, noticing the fact that for $\delta$ sufficiently small, the set $E\cap B(a,\delta^\epsilon\lambda)^c\cap \bm T^\sigma$ has size at least $\frac{1}{2}$ of the size of $E\cap \bm T^\sigma,$ we can replace $E$ by $E\cap B(a,\delta^\epsilon\lambda)^c$ in \eqref{eqK10} and get \eqref{eqK9}. See \cite{sogge} and Proposition 5.2 of \cite{miao} for details.

For the tube $\bm T^\delta$, we denote
$$\bm S^\delta=\bm T^\delta\cap E\cap\left\{x:\#\mathfrak{L}_{\theta,\sigma}(x,j)\ge2^{-2}N\left(\log_2\frac{1}{\delta^2}\right)^{-2}\right\}.$$
By the definition of $\mathfrak{L}_{\theta,\sigma}(x,j)$, we see that there is a $M_0\in(0,M]$ and a subcollection $\{T^\delta_{i_k}\}_{x}$ of $\{T^\delta_i\}_{i=0}^M$ that are in $\mathfrak{L}_{\theta,\sigma}(x,j)$ for each $x$, so that if we let $x$ run through every point in $\bm S^\delta$, and take the union of these subcollections to get $\{T^\delta_{i_k}\}_{k=1}^{M_0}$, then we will have
\[\sum\limits_{k=1}^{M_0}\chi_{T_{i_k}^\delta}\ge\frac{N}{2^2}\left(\log_2\frac{1}{\delta^2}\right)^{-2} \textrm{on}\  \bm S^\delta.\]
Recall that two $\delta$-tubes that intersect at angle $\theta$ would have intersection measure less than $C\frac{\delta^d}{\theta}$, so we have

\[|\bm S^\delta|\lesssim_\epsilon N^{-1}\int_{\bm T^\delta}\sum\limits_{k=1}^{M_0}\chi_{T_{i_k}^\delta}(x)dx\le N^{-1}\sum\limits_{k=1}^{M_0}|T_{i_k}^\delta\cap \bm T^\delta|\lesssim \frac{M_0\delta^d}{N\theta},\]
together with the simple fact
 \[|\bm S^\delta|\gtrsim_\epsilon \lambda|\bm T^\delta|,\]
we conclude
\begin{equation}\label{eqK11}M_0\gtrsim_\epsilon\theta\delta^{-1}N\lambda.\end{equation}

Now, consider the average of function $f=\chi_{E}$ over $T_{i_k}^\delta\cap\{y:\textrm{dist}(y,\gamma_j)\in[\sigma/2,\sigma)\},$ we have
 \[\delta^{-(d-1)}\int_{T_{i_k}^\delta}f(y)dy=\delta^{-(d-1)}\left|T_{i_k}^\delta\cap E\cap \left\{y:\textrm{dist}(y,\gamma_j)\in[\sigma/2,\sigma)\right\}\right|\gtrsim_\epsilon\lambda.\]
 On the other hand,
 \[\delta^{-(d-1)}\int_{T_{i_k}^\delta}f(y)dy\le A_{\delta}^{\theta,\sigma}(f)(\xi_{i_k}).\]
After combining these two inequalities, we square both sides, multiply $\delta^{d-1}$ and sum up with respect to $k=1,\ldots,M_0$, then we have
 \[\begin{aligned}M_0\delta^{d-1}\lambda^2&\lesssim_\epsilon\sum\limits_{k=1}^{M_0}|A_{\delta}^{\theta,\sigma}(f)(\xi_{i_k})|^2\delta^{d-1}\\&\lesssim\|A_{\delta}^{\theta,\sigma}(f)(\xi)\|_{L^2(S^{d-1})}^2\\&\lesssim_\epsilon\frac{\theta^{d-2}}{\sigma^{d-2}}|E\cap \{y:\textrm{dist}(y,\gamma_j)\in[\sigma/2,\sigma)\}|\\&\lesssim_\epsilon\frac{\theta}{\sigma\lambda^{d-3}}|E\cap \bm T^\sigma|,\end{aligned}\]
 where we used the maximality of the $\{\xi_k\}$, \eqref{eqK7} and \eqref{eqK12}. Using \eqref{eqK11} for the estimate of $M_0$, we get \eqref{eqK10}.
 \end{proof}

\subsection{Completion of the proof}
We give the estimate corresponding to high and low multiplicity cases separately, and we start with the simple one.
\lemma For N satisfy $\mathrm I$,
\begin{equation}\label{eqK14}|E|\gtrsim\frac{\lambda M\delta^{d-1}}{N}.\end{equation}
\proof Let $E_0=\{x\in E:\sum_{k=1}^M\chi_{T_k^\delta(x)}\le N\}$. Recalling that $N$ fulfills case I, we know $|T_i^\delta\cap E_0|\ge\lambda|T_i^\delta|/2$ for at least $M/2$ values of $i=i_k$. Thus
$$|E|\ge\left|\bigcup\limits_{k=1}^{M/2}(E_0\cap T_{i_k}^\delta)\right|\ge N^{-1}\sum\limits_{k=1}^{M/2}|E_0\cap T_{i_k}^\delta|\gtrsim\frac{\lambda M\delta^{d-1}}{N}.$$

In order to estimate the high multiplicity case, we need to establish a bush argument for the collection of hairbrushes $\{E\cap T_j^\sigma\}$, where the following lemma plays a key role.
\begin{lemma} Suppose there are $M$ tubes $\{T_j^\sigma\}_{j=1}^M$ such that $j\neq j'$ and $T_j^\sigma\cap T_{j'}^\sigma\neq\emptyset$ implies $\angle(T_j^\sigma, T_{j'}^\sigma)\ge\gamma$ for some $0<\gamma<\frac{\pi}{2}$. Assume also that for some $\rho>0$ and any $a\in\mathbb{R}^d$, there are $M_0$ such tubes satisfying 
\begin{equation}
\rho|T_j^\sigma|\le|T_j^\sigma\cap E\cap B(a,\sigma/\gamma)^c|.
\end{equation}
Then we have
\begin{equation}
|E|\ge\frac{\rho\sigma^{d-1}M_0^{1/2}}{2}.
\end{equation}\label{lemma4}
\end{lemma}
\begin{proof} By relabeling the indices, we have a sequence $\{T_j^\sigma\}^{M_0}_{j=1} $ satisfying 
$$\rho\sigma^{d-1}M_0\le\int_E\sum\limits_{j=1}^{M_0}\chi_{T_j^\sigma}(x)dx.$$
Thus, there exists an $x_0\in E$ such that
$$\sum_{j=1}^{M_0}\chi_{T_j^\sigma}(x_0)\ge\frac{\rho\sigma^{d-1}M_0}{2|E|}.$$
Noting that the diameter of $T_{j'}^\sigma\cap T_j^\sigma$ is at most $\sigma/\gamma$, so $B(x_0,\sigma/\gamma)^c \cap T^{\sigma}_j \cap T^{\sigma}_{j'}=\emptyset$, we have
$$|E|\ge\left|E\cap B(x_0, \sigma/\gamma)^c\cap \bigcup\limits_{\{j: x_0\in T_j^\sigma\}}T_j^\sigma\right|\ge\sum\limits_{\{j: x_0\in T_j^\sigma\}}|E\cap B(x_0, \sigma/\gamma)^c\cap T_j^\sigma|\ge\frac{\rho^2\sigma^{2(d-1)}M_0}{4|E|}.$$
\end{proof}
\begin{lemma} Let N satisfy $\mathrm {II}_{\theta,\sigma}$, then we have
\begin{equation}\label{eqK13}
|E|\gtrsim_\epsilon\lambda^{d+1}N(M\delta^{d-1})^{\frac{1}{d-1}}\delta^{d-2}\end{equation}
\end{lemma}
\begin{proof} By the multiplicity argument, we know that for some suitable constant $c$, there are at least 
$$[cM(\log_2\frac{1}{\delta^2})^{-2}]$$ many tubes in $\mathrm{II}_{\theta,\sigma}$, denote them by
$$\{T_j^\delta\}_{j=1}^{[cM(\log_2\frac{1}{\delta^2})^{-2}]}.$$ Let
$$\gamma=\frac{\sigma}{\delta^\epsilon\lambda},$$
then clearly $\gamma\ge\delta^{1-\epsilon}$. If $\gamma\ge\frac{\pi}{2}$, then \eqref{eqK13} follows directly from \eqref{eqK9}.
Otherwise, take a maximal $\gamma$-separated subset of $\{\xi_j\}_{j=1}^{[cM(\log_2\frac{1}{\delta^2})^{-2}]}$ and denote the size of this subset to be $M_0$. By maximality, we see easily
$$M_0\gtrsim\frac{M}{{\left(\log_2\frac{1}{\delta^2}\right)}^{2}}\delta^{d-1}\left(\frac{\delta^\epsilon\lambda}{\sigma}\right)^{d-1}\gtrsim_\epsilon M\delta^{d-1}\left(\frac{\lambda}{\sigma}\right)^{d-1},$$
and using \eqref{eqK9} one may easily check that if we let $\rho=C_\epsilon\lambda^d\sigma^{2-d}\delta^{d-2+\epsilon}N$ for some proper constant $C_\epsilon$ then all requirements of Lemma \ref{lemma4} are fulfilled, so we have
$$
\begin{aligned}
|E|&\gtrsim_\epsilon\lambda^d\sigma^{2-d}\delta^{d-2}N\cdot\sigma^{d-1} M_0^{\frac{1}{2}}\\
&\ge\lambda^d\sigma\delta^{d-2}N M_0^{\frac{1}{d-1}}\\
&\gtrsim_\epsilon\lambda^d\sigma\delta^{d-2}N\left({M\delta^{d-1}\left(\frac{\lambda}{\sigma}\right)^{d-1}}\right)^{\frac{1}{d-1}}\\
&=\lambda^d\sigma\delta^{d-2}N(\delta^{d-1}M)^{\frac{1}{d-1}}\lambda\sigma^{-1}\\
&\ge\lambda^{d+1}N(\delta^{d-1}M)^{\frac{1}{d-1}}\delta^{d-2},
\end{aligned}
$$
where we used the fact that $M_0^{1/2}\ge M_0^{1/(d-1)}$ since $M_0\ge1$ and $d\ge3$.
\end{proof}
Now if we take the geometric mean of \eqref{eqK13} and \eqref{eqK14}, we get \eqref{eqK3}, completing the proof.
\section{Nikodym-type maximal function in spaces of constant curvature}
Once we know how to prove Wolff's result without appealing to induction on scales, it is easy to generalize Sogge's result for Nikodym maximal function in 3-dimensional spaces of constant curvature to any dimension $d\ge3$. This section is parallel to the first half of our paper. Throughout this section, we fix a dimension $d\ge3$ and use $C$, $c$ to denote various constants that only depend on the curvature of the manifold.
\subsection{Preliminaries}
Let $(M^d,g)$ be a Riemannian manifold. Throughout the second half of our paper, we fix a number $\alpha>0$ that is smaller than $\min\{1,\frac{1}{2}\mathrm{inj}M^d\}$, where $\mathrm{inj}M^d$ denotes the injectivity radius of $M^d$. Let $\gamma_x$ denote any geodesic passing through $x\in M^d$ of length $\alpha$. Using the metric, we let
\[T_x^\delta=\{y\in M^d:\mathrm{dist}(y,\gamma_x)\le\delta\}\]
be a tubular $\delta-$neighborhood around $\gamma_x$. We shall also sometimes use the notation $T_{\gamma_x}^\delta$ to denote the same tube. Now given a function $f$ on $M^d$, we can define the Nikodym maximal function
\[f_\delta^{**}(x)=\sup\frac{1}{|T_x^\delta|}\int_{T_x^\delta}|f(y)|dy.\]

Since the Nikodym problem is local, Wolff's result(Theorem \ref{THM1}) implies if $M^d$ has constant curvature $0$, then we have
\[\|f_\delta^{**}\|_{L^{q}(M^d)}\lesssim_\epsilon\delta^{1-\frac{d}{p}}\|f\|_{L^p(M^d)},\ p=\frac{d+2}{2},\ q=(d-1)p'.\]

On the other hand, Sogge \cite{sogge} showed that bounds like this hold in the constant curvature case if $d=3$ (Theorem \ref{THM2}).

The main result of this section is to extend Sogge's result to any dimension $d\ge3.$

\begin{theorem} \label{THM4}Assume that $\mathrm(M^d,g\mathrm)$ has constant curvature. Then for $f$ supported in a compact subset $K$ of a coordinate patch and all $\epsilon>0$
\begin{equation}
\|f_\delta^{**}\|_{L^q(M^d)}\lesssim_\epsilon \delta^{1-\frac{d}{p}}\|f\|_{L^p(M^d)},
\end{equation}
where $1\le p\le\frac{d+2}{2},\ q=(d-1)p'$.
\end{theorem}

Clearly, the $L^1\rightarrow L^\infty$ bounds are trivial, so it suffices to prove the following restricted weak-type estimate

\begin{equation}\label{eqn1}
|\{x\in M^d:(\chi_E)_\delta^{**}(x)\ge\lambda\}|\lesssim_\epsilon(\lambda^{-\frac{d+2}{2}}\delta^{\frac{2-d}{2}}|E|)^{\frac{2d-2}{d}},
\end{equation}
where $E$ is a set contained in our coordinate patch.

Before turning to the proof of \eqref{eqn1}, we quote a useful geometric lemma which is essentially in \cite{soggem}.
\begin{lemma}
\label{lemman0}
Suppose $\gamma_1,\gamma_2$ are geodesics of length $\alpha$ and assume that the $\gamma_j$ belong to a fixed compact subset $K$ of $M^d$. Suppose also $a\in T_{\gamma_1}^\delta\cap T_{\gamma_2}^\delta$. Then there is a constant $c>0$, depending on $(M^d,g)$ and $K$, so if
\[\angle(T_{\gamma_1}^\delta, T_{\gamma_2}^\delta)\ge\frac{\delta}{c\lambda},\]
then we have
\[(T_{\gamma_1}^\delta\cap T_{\gamma_2}^\delta)\setminus B(a,\lambda)=\emptyset.\]
\end{lemma}

Here we are using the induced metric on the unit tangent bundle to define the angle between two geodesics(tubes) $\gamma_1,\gamma_2$ of length $\alpha$
\[\angle(T_{\gamma_1}^\delta, T_{\gamma_2}^\delta)=\angle(\gamma_1,\gamma_2)=\min\limits_{x_j\in\gamma_j,\tau_j=\gamma'_j|_{\gamma_j=x_j}}\mathrm{dist}_{UTM^d}((x_1,\tau_1),(x_2,\tau_2)).\]
Here $\gamma'_j|_{\gamma_j=x_j}$ denotes a unit tangent vector at $x_j$.

As in \cite{sogge}, \cite{wolff} and \cite{bourgain}, it is convenient to work with a discrete form of the problem. 

We fix a geodesic $\gamma_0$ and work in Fermi normal coordinates near $\gamma_0$. To obtain these Fermi normal coordinates, we first fix a point $x_0\in\gamma_0$ and then choose an orthonormal basis $\{e_k\}_{k=1}^{d}\subset T_{x_0}M^d$ with $e_1$ being a unit tangent vector of $\gamma_0$ at $x_0$. Using parallel transport, one propagates this basis to every point of $\gamma_0$. If we choose $\gamma_0(s)$ to be the arc length parameterization of $\gamma_0$ with $\gamma_0(0)=x_0$ and $\gamma_0'(0)=e_1$, then the resulting vectors $\{e_k(s)\}$ will be orthonormal in $T_{\gamma_0(s)}M^d$ and $\gamma'(s)=e_1(s)$. We then assign Fermi coordinates $(x_1,x_2,\ldots,x_d)=(x,x')$ to a point $x$, if it is the endpoint of the geodesic of length $|x'|$ starting at $\gamma_0(x_1)$ with tangent vector $(0,x')$. 

These coordinates provide us with some good properties. First, the rays $t\rightarrow(x_1,tx')$ are geodesics orthogonal to $\bm\gamma$. Second, by construction we see that the vector fields $\partial x_k$ are parallel along $\bm\gamma$. Also, these Fermi normal coordinates are unique up to rotations preserving the $x_1$-axis. See details in \cite{sogge}.

Now we fix a small number $c>0$, and consider only the geodesics $\gamma$ that, belong to the collection
\[G=\{\gamma_{x'}: (0,x')\in\gamma_{x'}\ \mathrm{for\ some}\ x', \angle(\gamma_{x'},\gamma_0)\le c\}.\]

Then for a large fixed constant $C_0$, we consider a $C_0\delta$-separated collection $\{x'_j\}^M_{j=1}$ of the set $\{(0,x')\in M^d:(\chi_E)_\delta^{**}(0,x')\ge\lambda\}$. For each $j$, we choose a tube $T_j^\delta$ to be the  $\delta$-tube about some $\gamma_{x_j'}\in G$ such that
\[|E\cap T_j^\delta|\ge\lambda|T_j^\delta|,\]
then \eqref{eqn1} would follow from the uniform bounds

\begin{equation}\label{eqn2}
M\delta^{d-1}\lesssim_\epsilon(\lambda^{-\frac{d+2}{2}}\delta^{\frac{2-d}{2}}|E|)^{\frac{2d-2}{d}},
\end{equation}

Indeed, this inequality implies the slightly stronger version of \eqref{eqn1}, where the left hand side is replaced by $|\{(0,x')\in M^d:(\chi_E)_\delta^{**}(0,x')\ge\lambda\}|$, and we replace the maximal operator by one involving averaging over $\delta$-tubes with central geodesics in $G$. 

Note since the basepoints $\{x'_j\}$ of the tubes are $\delta$-separated, we must have 
\[\angle(T_j^\delta,T_i^\delta)>c\delta,\]
for some constant $c$. Now we use the exact same multiplicity argument as the one we used for the Kakeya problem in $\mathbb R^d$.

\subsection{Multiplicity argument}

Consider parameters $\theta\in[\delta,1],$ and $\sigma \in [\lambda\delta, 1]$.
First, for $1\le j\le M$ and $x\in T_j^\delta$ fixed, let
$$\mathfrak{L}_\theta(x,j)=\{i:x\in T_i^\delta, \angle{(T_j^\delta, T_i^\delta)\in[\theta/{2},\theta)}\}$$
index the tubes $T_i^\delta$ containing $x$ which intersect the fixed tube $T_j^\delta$ at angle comparable to $\theta$. Next, let
$$\mathfrak{L}_\sigma(x,j)=\{i:x\in T_i^\delta, |T_i^\delta\cap\{y\in E: \mathrm{dist}(y,\gamma_j)\in[\sigma/{2},\sigma)\}|\ge(2^4\log_2\frac{1}{\delta^2})^{-1}\lambda|T_i^\delta|\}$$
index the tubes $T_i^\delta$ containing $x$ which intersect the fixed tube $T_j^\delta$ at $x$ such that there is non-trivial portion of $T_i^\delta\cap E$ with distance to $\gamma_j$ comparable to $\sigma$.
Now let
$$\mathfrak{L}_{\theta,\sigma}(x,j)=\mathfrak{L}_{\theta}(x,j)\cap\mathfrak{L}_{\sigma}(x,j).$$
Then we have the following
\begin{lemma}\label{lemman1}There are $N\in\mathbb{N}$ and $\theta\in[\delta, 1]$, $\sigma\in[\lambda\delta, 1]$ that fulfills the following two cases

$\mathrm{I}.$ $($Low multiplicity case$)$There are at least $M/2$ values of $j$ for which
\[\left|\left\{x\in T_j^\delta\cap E: \#\{i:x\in T_i^\delta\}\le N\right\}\right|\ge\frac{\lambda}{2}|T_j^\delta|.\]

$\mathrm{II}_{\theta,\sigma}.$ $($High multiplicity case at angle $\theta$ and distance $\sigma$$)$There are at least $M/(2(\log_2{1/\delta^2}))^2$ many values of $j$ for which
\begin{equation}\left|\left\{x\in T_j^\delta\cap E: \#\mathfrak{L}_{\theta,\sigma}\ge\frac{N}{(2\log_2\frac{1}{\delta^2})^2}\right\}\right|\ge\frac{\lambda}{(4\log_2\frac{1}{\delta^2})^2}|T_j^\delta|.\end{equation}
\end{lemma}
The proof is identical to that of Lemma \ref{lemma1}. We also have the same bound for $\sigma/\theta$ as in the remark of Lemma \ref{lemma1} for the same reason.
\begin{equation}\label{eqn3}\lambda\lesssim_\epsilon\frac{\sigma}{\theta}\lesssim 1.\end{equation}

\subsection{Auxiliary maximal function}

Throughout this section, we fix a tube $\bm T^\delta$. We follow Sogge's strategy in \cite{sogge} closely and generalize it to any dimension $d\ge3$. We work in the Fermi normal coordinates near the central geodesic $\bm\gamma$ of $\bm T^\delta$.

We now define the auxiliary maximal function for 

$$A_{\delta}^{\theta,\sigma}(f)(x')=\sup\limits_{T_{\gamma_{x'}}\in S_{x'}}\frac{1}{|T_{\gamma_{x'}}^\delta|}\int_{T_{\gamma_{x'}}^\delta\cap \{y:|y'|\in[\sigma/2,\sigma]\}}|f(y)|dy,$$
where  the supremum runs through the collection of tubes 
$$S_{x'}=\{T^\delta_{\gamma_{x'}}:(0,x')\in\gamma_{x'},\ \gamma_{x'}\cap\bm\gamma\neq\emptyset,\angle(\gamma_{x'},\bm\gamma)\in[\theta/2,\theta)\},$$
 and define $A_{\delta}^{\theta,\sigma}(f)({x'})$ to be zero if $S_{x'}=\emptyset.$
\begin{theorem} With $A_{\delta}^{\theta,\sigma}$ as above, then we have
\begin{equation}\label{eqn9}\|A_{\delta}^{\theta,\sigma}(f)\|_{L^2}\lesssim\left(\log\frac{1}{\delta}\right)^\frac{1}{2}\left(\frac{\theta}{\sigma}\right)^\frac{d-2}{2}\|f\|_{L^2}.\end{equation}

\end{theorem}
\begin{proof} Write $A_{\delta}^{\theta,\sigma}(f)$ simply as $A(f)$. The proof is very similar to the proof of Theorem \ref{theorem1}. We estimate the integral
$$\int|A(f)|^2({x'})dx.$$

Noticing that if we require $S_{x'}\neq\emptyset$, then $|x'|\le C\sin\theta$ for some $C$ that only depends on the curvature. We define the subset $C_\theta $ in the base hyperplane $\{x\in(M^d,g): x_1=0\}$ by
$$C_\theta=\{x':|x'|\le C\sin\theta\}.$$

\begin{figure}
  \centering
    \includegraphics[height=8cm]{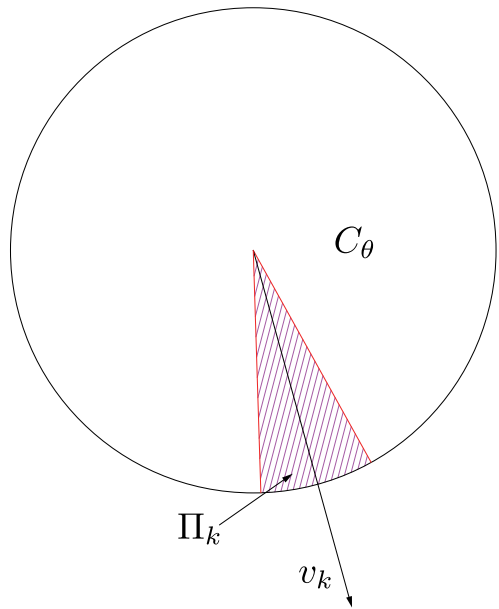}
  \caption{$\Pi_k$ in the base hyperplane}
  \label{4}
\end{figure}

Take a maximal $\frac{\delta}{\sin\theta}$-separated subset $\{v_k\}$ of $S^{d-2}$, which has cardinality comparable to $[(\frac{\theta}{\delta})^{d-2}]$. Let $\Pi_k$ be the conic set in $\{x:x_1=0\}$ such that 
$$\Pi_k\cap \sin\theta\cdot S^{d-1}=B(\sin\theta\cdot v_k,\delta)\cap \sin\theta\cdot S^{d-1},$$ see Figure \ref{4}. As in proof of Theorem \ref{theorem1}, we must have $\cup_k\Pi_k\supset C_\theta$. And by the maximality of $\{v_k\}$, we can further assume $\Pi_k$'s to be pairwise disjoint.

Consider $x'\in\Gamma_k$ for some k. Let 
$$H_k=\textrm{span}\{e_1,(0,v_k)\},$$
Then $H_k$ would be totally geodesic as a Fermi 2-plane. Remember that we require $\bm \gamma\cap\gamma_{x'}\neq\emptyset$, so any tube $T_{\gamma_{x'}}^\delta\in S_{x'}$ must lie in a $C\delta$-neighborhood $H_k^{C\delta}$. Where $C$ is again some suitable constant that only depends on the curvature.
Let
$$V_k=\{x:|x_1|\le1\}\cap H_k^{C\delta},$$
then by the remark of Lemma \ref{lemma2}, we have
$$\sum\limits_k\chi_{V_k\cap \{y:|y'|\in[\sigma/2,\sigma)\}}(y)\lesssim \left(\frac{\theta}{\sigma}\right)^{d-2}.$$

Similar to the Kakeya case in $\mathbb{R}^d$, we conclude using the above fact and a twofold application of Schwarz's inequality, the theorem would follow from the following $L^2$ estimate for each $k$,
\begin{equation}\label{eqn8}\|A(f\chi_{V_k})\|_{L^2(\Pi_k)}\lesssim\left(\log\frac{1}{\delta}\right)^\frac{1}{2}\|f\chi_{V_k}\|_{L^2}.\end{equation}

To prove \eqref{eqn8}, we need a curved version of the 2-dimensional Nikodym maximal inequality.

To state it we now suppose that $(M^2,g)$ is a 2-dimensional Riemannian manifold. If we fix a geodesic $\gamma_0\subset M^2$ of length $\alpha\le\min\{1,(\mathrm{inj}M^2)/2\}$, we consider all geodesic $\{\gamma\}$ of this length which are close to $\gamma_0$. Let $\gamma_1(t)$ be a geodesic which intersects $\gamma_0$ orthogonally and is parameterized by arc length. We set
\begin{equation}
M_{\delta} g(t)=\sup\limits_{\gamma_1(t)\in\gamma}\delta^{-1}\int_{\{y:\mathrm{dist}(y,\gamma)\le\delta\}}|g(y)|dy.
\end{equation}
We claim \eqref{eqn8} would follow from
\begin{equation}\label{eqn10}
\|M_{\delta} g\|_{L^2(dt)}\lesssim\left(\log\frac{1}{\delta}\right)^\frac{1}{2}\|g\|_{L^2(M^2)}.
\end{equation}
This is (2.43) in \cite{sogge}, and we refer readers to \cite{sogge} and \cite{soggem} for the proof.

Now we show how \eqref{eqn10} implies \eqref{eqn8}. We use the same trick as we did for the Kakeya problem in Euclidean case. Without loss of generality, we fix $k$, assume $e_2=(0,v_k)$ and only consider functions $f$ with support contained in $V_k\cap\{y:|y'|\in[\sigma/2,\sigma)\}$. 

\begin{figure}
  \centering
    \includegraphics[height=8cm]{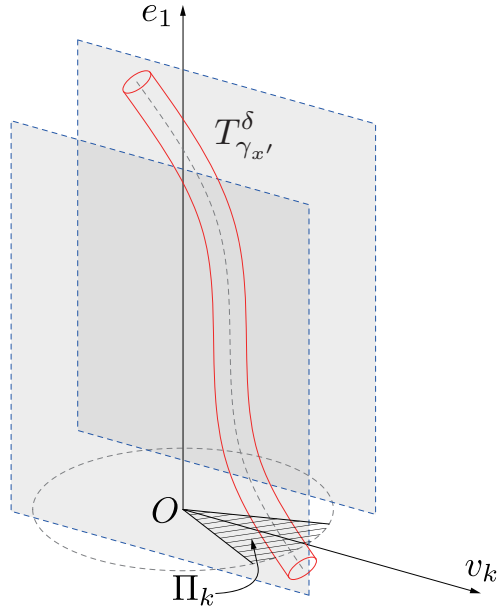}
  \caption{$T_{\gamma_{x'}}^\delta$ contained in $V_k$}
  \label{5}
\end{figure}

Let $\mathbf{P}(s)$ be the surface which corresponds to the 2-plane $\{y\in (M^d,g):y=(y_1,y_2,s)\}$ with volume element $dV_s$, where $s$ is a $(d-2)$-dimensional parameter for the collection of those 2-planes with $|s|\le C\delta$. Since $\mathbf{P}(0)=\mathrm{span}\{e_1,e_2\}$ is a totally geodesic 2-plane and we are in constant curvature case, $|dV_0|\sim |dy_1dy_2|$.

For any $x=(0,x')=(0,x_2,x'')\in\Pi_k$, we consider the integral over the cross section $\mathbf{P}(s)\cap T_{\gamma_{x'}}^\delta$. Clearly, the projection of this cross section on to $\mathbf{P}(0)$ is contained in $\mathbf{P}(0)\cap T_{\gamma_{x'}}^{C'\delta}$ for some constant $C'$. Noticing the fact that $dV_s$ varies smoothly with respect to $s$, we see that for fixed $s$ with $|s|\le C\delta$
\[\int_{\mathbf{P}(s)\cap T_{\gamma_{x'}}^\delta}|f(y_1,y_2,s)|dV_s\lesssim\int_{\mathbf{P}(0)\cap T_{\gamma_{x'}}^{C'\delta}}|f(y_1,y_2,s)|dV_0\lesssim\int_{\mathbf{P}(0)\cap T_{\gamma_{x'}}^{C'\delta}}|f(y_1,y_2,s)|dy_1dy_2.\]
Since $\mathbf{P}(0)$ is totally geodesic, $\mathbf{P}(0)\cap T_{\gamma_{x'}}^{C'\delta}$ is contained in $\mathbf{P}(0)\cap T_{\gamma_{(0,x_2)}}^{C''\delta}$ for some $\gamma_{(0,x_2)}$ and $C''$. Then we have
 $$\begin{aligned}
\delta^{-(d-1)}\int_{T_{\gamma_{x'}}^\delta}|f(y)|dy&=\delta^{-(d-1)}\int_{|y''|\le C\delta}dy''\int_{\mathbf{P}(y'')\cap T_{\gamma_{x'}}^\delta}|f(y_1,y_2,y'')|dV_{y''}\\
&\le\delta^{-(d-1)}\int_{|y''|\le C\delta}dy''\int_{\mathbf{P}(0)\cap T_{\gamma_{x'}}^{C'\delta}}|f(y_1,y_2,y'')|dy_1dy_2\\
&\le\delta^{-(d-1)}\int_{|y''|\le C\delta}dy''\int_{\mathbf{P}(0)\cap T_{\gamma_{(0,x_2)}}^{C''\delta}}|f(y_1,y_2,y'')|dy_1dy_2\\
&\lesssim\delta^{-(d-2)}\int_{|y''|\le C\delta}M_\delta(f(\ldots,y''))(x_2)dy'',
\end{aligned}$$

Therefore,
\[A(f)(x')\lesssim\delta^{-(d-2)}\int_{|y''|\le C\delta}M_\delta(f(\ldots,y''))(x_2)dy''.\]

Integrating over $x_1, x_2$ and using Minkowski's inequality, we get
$$\begin{aligned}
&\ \ \ \ \left(\int_{|x_2|\le 1}|A(f)(x')|^2dx_2\right)^\frac{1}{2}\\
&\lesssim\delta^{-(d-2)}\int_{|y''|\le C\delta}dy''\left(\int_{|x_2|\le 1}|M_\delta(f(\ldots,y''))|^2(x_2)dx_2\right)^{\frac{1}{2}}\\
&\lesssim\left(\log\frac{1}{\delta}\right)^\frac{1}{2}\delta^{-(d-2)}\int_{|y''|\le C\delta}\|f(\ldots,y'')\|_{L^2(y_1,y_2)}dy''\\
&\lesssim\left(\log\frac{1}{\delta}\right)^\frac{1}{2}\delta^{-\frac{(d-2)}{2}}\|f\|_{L^2}.
\end{aligned}$$

Noticing $|x''|\lesssim \delta$ for $x\in V_k$, this leads to \eqref{eqn8}, so the proof is complete.
\end{proof}

\subsection{A key lemma}
This section is parallel to section 2.4. From now on, let $N$ be the number that fulfills both case I and $\mathrm{II}_{\theta,\sigma}$, and again we fix a index $j$ such that $\bm T^\delta=T_{j}^\delta$ satisfy $\mathrm{II}_{\theta,\sigma}$. Using our $L^2$ estimate for the auxiliary maximal function, we will show that we can generalize Proposition 2.5 in \cite{sogge} to any dimension $d\ge3$.

\begin{lemma}\label{lemman2}For any $\epsilon>0$, any point $a$
\begin{equation}\label{eqn6}|E\cap B(a,\delta^\epsilon\lambda)^c\cap \bm T^\sigma|\gtrsim_\epsilon\lambda^d N\sigma\delta^{d-2}.\end{equation}
\end{lemma}
\begin{proof} Clearly, it suffices to prove
\begin{equation}\label{eqn4}|E\cap \bm T^\sigma|\gtrsim_\epsilon\lambda^d N\sigma\delta^{d-2}.\end{equation}
For the tube $\bm T^\delta$, we denote
$$\bm S^\delta=\bm T^\delta\cap E\cap\left\{x:\#\mathfrak{L}_{\theta,\sigma}(x,j)\ge2^{-2}N\left(\log_2\frac{1}{\delta^2}\right)^{-2}\right\}.$$
By the definition of $\mathfrak{L}_{\theta,\sigma}(x,j)$, we see that there is a $M_0\in(0,M]$ and a subcollection $\{T^\delta_{i_k}\}_{x}$ of $\{T^\delta_i\}_{i=0}^M$ that are in $\mathfrak{L}_{\theta,\sigma}(x,j)$ for each $x$, if we let $x$ run through every point in $\bm S^\delta$, and take the union of these subcollections to get $\{T^\delta_{i_k}\}_{k=1}^{M_0}$, then we will have
\[\sum\limits_{k=1}^{M_0}\chi_{T_{i_k}^\delta}\ge\frac{N}{2^2}\left(\log_2\frac{1}{\delta^2}\right)^{-2} \textrm{on}\  \bm S^\delta.\]
It follows from Lemma \ref{lemman0} that two $\delta$-tubes intersect at angle comparable to $\theta$ have intersection measure like $\frac{\delta^d}{\theta}$, so we have

\[|\bm S^\delta|\lesssim_\epsilon N^{-1}\int_{\bm T^\delta}\sum\limits_{k=1}^{M_0}\chi_{T_{i_k}^\delta}(x)dx\le N^{-1}\sum\limits_{k=1}^{M_0}|T_{i_k}^\delta\cap \bm T^\delta|\lesssim \frac{M_0\delta^d}{N\theta}.\]
Together with the simple fact
 \[|\bm S^\delta|\gtrsim_\epsilon \lambda|\bm T^\delta|,\]
we conclude
\begin{equation}\label{eqn5}M_0\gtrsim_\epsilon\theta\delta^{-1}N\lambda.\end{equation}

Now, consider the average of the function $f=\chi_{E}$ over $T_{i_k}^\delta\cap\{y:\textrm{dist}(y,\gamma_j)\in[\sigma/2,\sigma)\}$
 \[\delta^{-(d-1)}\int_{T_{i_k}^\delta}f(y)dy=\delta^{-(d-1)}|T_{i_k}^\delta\cap E\cap \{y:\textrm{dist}(y,\gamma_j)\in[\sigma/2,\sigma)\}|\gtrsim_\epsilon\lambda,\]
 On the other hand, for some large $C$ that only depends on curvature and some $y'_{i_k}$ that lies in the $C_0\delta$-neighborhood of $x'_{i_k}$, we have
 \[\delta^{-(d-1)}\int_{T_{i_k}^\delta}f(y)dy\le A_{C\delta}^{\theta,\sigma}(f)(y'_{i_k}),\]
 
After combining these two inequalities,  we square both sides, multiply $\delta^{d-1}$ and sum up with respect to $k=1,\ldots,M_0$, arriving at
 \[\begin{aligned}M_0\delta^{d-1}\lambda^2&\lesssim_\epsilon\sum\limits_{k=1}^{M_0}|A_{C\delta}^{\theta,\sigma}(f)(y'_{i_k})|^2\delta^{d-1}\\
 &\lesssim\|A_{C\delta}^{\theta,\sigma}(f)(y')\|_{L^2}^2\\
 &\lesssim_\epsilon\frac{\theta^{d-2}}{\sigma^{d-2}}|E\cap \{y:\textrm{dist}(y,\gamma_j)\in[\sigma/2,\sigma)\}|\\&\lesssim_\epsilon\frac{\theta}{\sigma\lambda^{d-3}}|E\cap \bm T^\sigma|,\end{aligned}\]
 where we used the maximality of the $\{x'_k\}$, \eqref{eqn3} and \eqref{eqn9}. Using \eqref{eqn5} for the estimate of $M_0$, we get \eqref{eqn4}.
 \end{proof}
\subsection{Completion of the proof}
Again, we give the estimate corresponding to the high and low multiplicity cases separately.

As what happened in the Euclidean case, if $N$ satisfy $\mathrm I$, it's easy to see that
\begin{equation}\label{eqn7}|E|\gtrsim\frac{\lambda M\delta^{d-1}}{N}.\end{equation}

In order to estimate the high multiplicity case, we need to use a curved version of the bush argument, which is basically the following lemma(\cite{soggem}):
\begin{lemma} Suppose there are $M$  tubes $\{T_j^\sigma\}_{j=1}^M$ such that $j\neq j'$ and $T_j^\sigma\cap T_{j'}^\sigma\neq\emptyset$ implies $\angle(T_j^\sigma, T_{j'}^\sigma)\ge C\gamma$ for some $0<\gamma<1$. Assume also that for some $\rho>0$ and any $a\in\mathbb{R}^d$, there are $M_0$ such tubes satisfying 
\begin{equation}
\rho|T_j^\sigma|\le|T_j^\sigma\cap E\cap B(a,\sigma/\gamma)^c|.
\end{equation}
Then if $C$ is large enough, we have
\begin{equation}
|E|\gtrsim \rho\sigma^{d-1}M_0^{1/2}.
\end{equation}
\label{lemman3}
\end{lemma}
By Lemma \ref{lemman0}, the diameter of $T_{j'}^\sigma\cap T^\sigma_{j}$ is like $\sigma/\gamma$, thus the proof of this lemma is identical to the proof of Lemma \ref{lemma4}.

Finally, we estimate the high multiplicity case to finish the proof.
\begin{lemma} Let N satisfy $\mathrm {II}_{\theta,\sigma}$. Then we have
\begin{equation}\label{eqn7}
|E|\gtrsim_\epsilon\lambda^{d+1}N(M\delta^{d-1})^{\frac{1}{d-1}}\delta^{d-2}\end{equation}
\end{lemma}
\begin{proof} By the multiplicity argument, we know that for some suitable constant $c$, there are at least 
$$[cM(\log_2\frac{1}{\delta^2})^{-2}]$$ many tubes in $\mathrm{II}_{\theta,\sigma}$, denote them by
$$\{T_j^\delta\}_{j=1}^{[cM(\log_2\frac{1}{\delta^2})^{-2}]}.$$
Let
$$\gamma=\frac{\sigma}{\delta^\epsilon\lambda}.$$
Then clearly $\gamma\ge\delta^{1-\epsilon}$. If $\gamma\ge 1$, then \eqref{eqn7} follows directly from \eqref{eqn6}.
Otherwise, take a maximal $\gamma$-separated subset of $\{x'_j\}_{j=1}^{[cM(\log_2\frac{1}{\delta^2})^{-2}]}$ and denote the total number of this subset to be $M_0$. By maximality, we see easily
$$M_0\gtrsim\frac{M}{\left({\log_2\frac{1}{\delta^2}}\right)^{2}}\delta^{d-1}\left(\frac{\delta^\epsilon\lambda}{\sigma}\right)^{d-1}\gtrsim_\epsilon M\delta^{d-1}\left(\frac{\lambda}{\sigma}\right)^{d-1}$$
and using \eqref{eqn6} one may easily check that if we let $\rho=C_\epsilon\lambda^d\sigma^{2-d}\delta^{d-2+\epsilon}N$ for some proper constant $C_\epsilon$ then all requirements of Lemma \ref{lemman3} are fulfilled, so we have
$$
\begin{aligned}
|E|&\gtrsim_\epsilon\lambda^d\sigma^{2-d}\delta^{d-2}N\cdot\sigma^{d-1} M_0^{\frac{1}{2}}\\
&\ge\lambda^d\sigma\delta^{d-2}N M_0^{\frac{1}{d-1}}\\
&\gtrsim_\epsilon\lambda^d\sigma\delta^{d-2}N\left({M\delta^{d-1}\left(\frac{\lambda}{\sigma}\right)^{d-1}}\right)^{\frac{1}{d-1}}\\
&=\lambda^d\sigma\delta^{d-2}N(\delta^{d-1}M)^{\frac{1}{d-1}}\lambda\sigma^{-1}\\
&\ge\lambda^{d+1}N(\delta^{d-1}M)^{\frac{1}{d-1}}\delta^{d-2},
\end{aligned}
$$
where we used the fact that $M_0^{1/2}\ge M_0^{1/(d-1)}$ since $M_0\ge1$ and $d\ge3$.
\end{proof}
Now if we take the geometric mean of \eqref{eqn5} and \eqref{eqn7}, we get \eqref{eqn2}, completing the proof of Theorem \ref{THM4}.

\section{Acknowledgement}
I would like to thank Professor C. Sogge for his guidance and patient discussions during my study. This paper would not have been possible without his generous support. It's a pleasure to thank my colleagues D. Ginsberg, S. Wang, X. Wang and S. Yu for many helpful discussions. I also would like to thank C. Miao and his group for going through an early draft of this paper. Special thanks to L. Jiang for the excellent figures.

\bibliography{paper}
\bibliographystyle{plain}

\end{document}